\documentclass{llncs}

\usepackage{amsmath}
\usepackage{amssymb}
\usepackage{graphicx}
\usepackage{xcolor}
\usepackage{nameref}
\usepackage{authblk}
\usepackage[mathscr]{eucal}

\usepackage{hyperref}
\definecolor{lgreen}{rgb}{0.0, 0.48, 0.0}
\definecolor{lpurple}{rgb}{0.48, 0.0, 0.48}
\hypersetup{linktocpage,
			colorlinks=true,
			linkcolor=lgreen,
			citecolor=lpurple}

\def \E {\mathbf{E}}
\def \P {\mathbf{P}}
\def \R {\mathbb{R}}
\def \L {L}
\def \vec {\boldsymbol}
\def \btheta {\vec{\theta}}
\def \bPsi {\vec{\Psi}}
\def \X {\mathcal{X}}
\def \bxi {\vec{\xi}}
\def \z {\mathtt z}
\def \r {\mathtt r}
\def \tilde {\widetilde}
\def \bar {\overline}
\def \a {\mathfrak a}
\def \c {\mathfrak c}

\def \Var {\mathrm{Var}\,}
\def \eps {\varepsilon}

\numberwithin{equation}{section}

\begin{document}
	
\title{Towards model selection for local log-density estimation. Fisher and 
Wilks-type theorems.\thanks{This work was done during the Master program at 
MIPT under supervision of Vladimir Spokoiny. I am very grateful to him for his 
support during my university studies.}}
\author{Sergey Dovgal \inst{1}\textsuperscript{,} \inst{2}}
\institute{Moscow Institute of Physics and Technology, \\9 Institutskiy per., 
Dolgoprudny, Moscow Region, 141700, Russian Federation;\\
\and
Institute for Information Transmission Problems (RAS),\\
Bolshoy Karetny per. 19, buld.1, Moscow 127051, Russian Federation;\\
\email{vit.north-at-gmail.com, dovgal-at-phystech.edu}}
\date{}
\maketitle

\begin{abstract}
The aim of this research is to make a step towards providing a tool for model 
selection for log-density estimation. The author revisits the procedure for 
local log-density estimation suggested 
by Clive Loader (1996) and extends the theoretical results to finite-sample 
framework with the help of machinery of Spokoiny (2012). The results include 
bias expression from ``deterministic'' 
counterpart and Fisher and Wilks-type theorems from ``stochastic''.
We elaborate on bandwidth trade-off \( h(n) = \arg\min O(h^p) + 
O_p(1/\sqrt{nh^d}) \) with explicit constants at big O notation.

Explicit expressions involve (i) true density 
function and (ii) model that is selected 
(dimension, bandwidth, kernel and basis, e.g. polynomial). Existing 
asymptotic properties directly follow from our results. From the expressions 
obtained it is possible to control ``the curse of dimension'' both from the 
side of log-density smoothness and the inner space dimension.
\end{abstract}		
		
\section{Introduction}

There is a famous trade-off between the parameters 
of the model: bandwidth, polynomial degree, the basis set, the kernel function. 
In the \textit{linear kernel density estimation} procedure (Parzen--Rozenblatt) 
\cite{Tsybakov}, the choice of the kernel function is very 
important for asymptotic rates. For example, if one introduces a \textit{risk} 
at point \( x_0 \) for given density estimator, 
then one can state the existence of \textit{minimax estimator}, which requires 
some special kernels (for example based on Legendre polynomials for quadratic 
risk) and particular 
dependence \( h = h(n) \) in order to minimize the risk.

Loader's procedure which we consider, has its advantages and disadvantages. Its 
main disadvantage is its computational complexity: in order to compute the 
estimate, we need to implement a convex 
optimization procedure, where each step requires numerical computing of some 
multidimensional integral. However, they have implemented a \texttt{locfit} 
R-package, and we refer to \cite{Loader} for their experimental results. 
Advantage of the procedure is that regardless of the kernel function, 
this estimator always provides the minimax optimal rates, the same as for 
respective linear estimators with special kernels (they are referred to as 
\textit{kernels 
of order \( p \)} in \cite{Tsybakov}). Its second advantage is that (in case of 
polynomial basis) we estimate the derivatives of log-density in addition to the 
value of log-density.

Another reason for developing 
finite-sample bounds for this particular estimator, is the use of the 
\textit{quasi-likelihood} concept: we were able to apply ideas of Spokoiny 
\cite{Spokoiny} for finite-sample estimation.

This study allowed to choose the ``best kernel'' according to our finite-sample 
bounds. In the case of pointwise estimation, the answer is probably the 
\textit{indicator kernel}, but it is still unclear, whether it 
is the same for uniform bounds for multi-point density estimation. 
The best bandwidth can be chosen by the familiar 
expression
\begin{equation}
	h(n) = \arg\min_h \left( O(h^p) + O_p\big((nh^d)^{-1/2}\big) \right) 
	\enspace ,
\end{equation}
where \( p \) stands for smoothness and \( d \) for dimension, \( n \) is a 
sample size. Since we provide explicit constants, it becomes possible to choose 
this minimum explicitly.

We also point out that despite the work that has been done, it is still not 
enough to provide data-driven procedure for construction of confidence 
intervals or confidence bands. Usually Fisher and Wilks theorems are used to 
construct confidence interval at point or a confidence band at some region \( x 
\in I \), 
but in order to choose the correct data-driven quantile function, 
\textit{bootstrap} provides a substantial (asymptotic and non-asymptotic) 
refinement in comparison with more conservative tools. Hopefully, the future 
research will give answers to these questions.

\subsection{Key objects and estimation procedure}
\label{section:objects}

Loader \cite{Loader} considers the following idea. Suppose the data \(X_1, X_2, 
\ldots, X_n \in \R^d \) is observed, where \( X_i \) are i.i.d. drawn from 
density function \( f \). Our goal is to construct the estimate \( \hat f(x_0) 
\) of the unknown density at given point \(x_0 \). The ordinary likelihood for 
density function is defined by equation $L(f) = \prod\limits_{i=1}^n f(X_i)$. 
We restrict ourselves to density functions that satisfy $\int_{\R} f(x) dx = 1$.
However, maximum for this likelihood over functions \( f \), is attained at sum 
of delta-functions~--- this is the reason why we impose further smoothness 
restrictions.

Note that the expectation of likelihood has the nice 
property of having maximum in true density function \( f^{\ast} \):
\begin{equation}
    f^\ast = \arg\max_{f \in L_2} \E L(f) \enspace .
\end{equation}

Let us change the procedure in the way so it can become more practical: 
consider (i) localization \( [x_0 - h, x_0 + h] \) with the change of variables 
\( t = \dfrac{x - x_0}{h} \) and (ii) choose some finite 
basis \( \psi_0(t), 
\ldots, \psi_{p-1}(t) \) for the 
unknown log-density function in the interval \( t \in [-1, 1] \) (or the cube 
\( 
[-1, 1]^{d} \) in case of multidimensional estimation). Let us also
introduce \textit{log-likelihood function} parametrized by some vector \( 
\btheta \):
\begin{equation}
	\label{eq:likelihood_functional}
	L(\btheta; \vec X, x_0, h) = \sum_{i=1}^n K_i \bPsi_i^\top \btheta - n \int 
	K \exp(\bPsi^\top \btheta) dx \enspace ,
\end{equation}
where \( \bPsi_i = (\psi_0(T_i), \psi_1(T_i), \ldots, \psi_{p-1}(T_i))^{\top} 
\), \( \btheta = (\theta_0, \theta_1, \ldots, 
\theta_{p-1}) \), \( K_i = K(T_i) \), \( T_i = \dfrac{X_i - x_0}{h} \), \( dx = 
h^d dt \). The 
principal example in the current article
will be the case of one-dimensional local polynomial estimation with an 
indicator kernel, where 
\( \psi_k(t) = t^k \), \( K(t) = \left[ -1 \leq t \leq 1 \right] \). 
We also discuss generalizations to the \( d \)-dimensional case throughout this 
article.

The motivation for the functional (\ref{eq:likelihood_functional}) is the 
following: first terms stands for 
basis approximation of given density function, and the second terms stands for 
Lagrange-type penalty.

Then we define \( \tilde{\btheta} \)~--- \textit{maximum likelihood estimator}, 
\( \btheta^{\ast} \)~--- \textit{target biased parameter}:
\begin{eqnarray}
	\tilde{\btheta}   &=& \arg\max_{\btheta} L(\btheta), \quad
	\btheta^{\ast} = \arg\max_{\btheta} \E L(\btheta) \enspace ,
\end{eqnarray}
We will also define \textit{target unbiased prameter} \( \btheta^{\bullet}(h) 
\) later through 
\textit{small bias condition}. Since true log-density function isn't 
neccessarily equal to the finite sum of basis functions, in practice one can 
choose any known approximation as an unbiased parameter. As an illustration, in 
the case of one-dimensional local polynomial estimation, unbiased parameter can 
be chosen as 
first \( p \) terms in Taylor expansion of \( \log f(x) \) near \( x_0 \):
\begin{eqnarray}
	\theta^{\bullet}_j(h) &=& \left. \dfrac{h^j}{j!} \dfrac{\partial^j \log 
	f(x)}{\partial x^j} \right|_{x = x_0} \enspace , \quad j = 0,1,\ldots, p-1 
	\enspace .
\end{eqnarray}
We will require that the first element of the basis is constant, \( \psi_0 
\equiv 1 \), so that \( \theta_0 \) usually corresponds to the sought-for 
log-density: \( \theta_0^{\bullet} = \log f(x_0) \). If the elements of the 
basis are linearly dependent, then it is not possible to perform the estimation 
procedure. During the proofs we will use auxilliary parameter defined by
\begin{equation}
	\btheta^{\circ} = (\theta_0^{\bullet}, 0, \ldots, 0) = (\log f(x_0), 0, 
	\ldots, 0) \enspace .
\end{equation}

Below we introduce the objects from finite-sample theory of Spokoiny: the 
\textit{information matrix} \( D_n^2 \), the \textit{score vector} \( \bxi \) 
and the \textit{variance matrix} \( V_n^2 \). The index \( n \) stands for 
sample size, though, in all the statements the value \( nh^d \) will be used as 
an \textit{effective sample size}.

\begin{equation}
\begin{aligned}
D_n^2 &= - \nabla^2 \E L(\btheta^{\ast}) \enspace ,\\
\bxi &= D_n^{-1} \nabla L(\btheta^\ast) \enspace ,
\end{aligned}
\qquad
\begin{aligned}
D_n^2(\btheta) &= - \nabla^2 \E L(\btheta) \enspace ,\\
V_n^2 &= \Var ( \nabla L(\btheta^\ast) ) \enspace .
\end{aligned}
\end{equation}

In general case of multidimensional density estimation, \( \X \subseteq \R^d \) 
we denote
\begin{equation}
	d_0^2(\btheta) = (nh^d)^{-1} D_n^2 (\btheta) \enspace ,
\end{equation}
where matrix \( d_0^2(\btheta) \) doesn't depend on \( n, h \).

Since stochastic part of \( L \) is linear on \( \btheta \), the stochastic 
part of gradient \( \nabla L \) doesn't 
depend on the argument: \( \nabla L(\btheta^\ast) - \E \nabla L(\btheta^\ast) 
\equiv \nabla L(\btheta) - \E \nabla L(\btheta) \). 
We also condsider matrix \( V_n^2(f) \), which 
describes 
the variance under the true measure \( f(x) \), depending on various functions 
\( 
f \):
\begin{equation}
	V_n^2(f) = \mathrm{Var}_{f}\; (\nabla L) \enspace .
\end{equation}

We also introduce the concentration neigbourhood for \( \tilde \theta \):
\begin{equation}
	\Theta_n(\z)
	=
	\Big\{ \btheta \colon
	\|D_{n}(\btheta - \btheta^\ast)\|
	\leq 
	\r_0(\z)
	\Big\} \enspace .
\end{equation}
The concentration radius \( \r_0(\z) \) will be described below. In particular, 
the concentration condition \nameref{condition:C} describes upper bounds 
for \( 
\r_0 \), and this condition is checked in section \ref{section:check}.

The quantile function for \( \chi^2 \)-like distributions is given by lemma 
\ref{lemma:quadratic}:
\begin{equation}
	\zeta(p, \z) = 2 \a \nu_0 (\sqrt{p} + \sqrt{2\z})
	, \quad
	\z \leq \mathfrak{g}^2 / 4 \enspace ,
\end{equation}
where the constants \( \a, \nu_0 \) are given by the conditions 
\nameref{condition:I} and \nameref{condition:ed_0}.

\subsection{Structure of the article}
In order to prove theorems \ref{theorem:concentration}, \ref{theorem:fisher} 
and \ref{theorem:wilks} we 
need to apply finite-sample machinery of Spokoiny \cite{Spokoiny}, and then 
check all the conditions. Thus, we state these conditions in section 
\ref{section:conditions} and explain what they mean. This is referred to as 
``level 1''. When we check these 
conditions, we want to express them in terms of the model (basis, dimension, 
badwidth, kernel and smoothness) and in terms of some unknown true variables: 
density value at the point \( x_0 \), oscillation and bias. This is referred to 
as ``level 2'', since our logic is clearly separated into layers.

After some preparation in form of conditions and constants we state and prove 
the main theorems in section \ref{section:theorems}. These theorems are quite 
general and are applicable for a wide range of models, see 
\cite{SpokoinyZhilova,Zhilova}. Some constants are separated from the 
formulation of the main theorems to keep the presentation more clear. The final 
expressions can be obtained combining the theorems and results from section 
\ref{section:check}, where we check the conditions in form of lemmas. The 
theorem \ref{theorem:accurate_small_bias} is a separate result, and doesn't 
follow from the general theory, so it is applicable only for log-density 
estimation procedure.

The results require some tools from linear algebra, technical lemmas on small 
bias and one result for deviation bounds for quadratic forms, which are 
presented in appendix.

\section{Conditions, Level 1}
\label{section:conditions}

We introduce four conditions: \nameref{condition:C}, \nameref{condition:I}, 
\nameref{condition:l_0}, \nameref{condition:ed_0}, according to Spokoiny. These 
conditions are used to prove 
the theorems in the section \ref{section:theorems}. We are going to check these 
conditions in the section \ref{section:check} after formulation of the main 
results.

\subsection{Identification Condition}
\begin{description}
	\item[\( (\vec{\mathcal I}) \) \label{condition:I}]
	There exists a constant \( \a > 0 \) such that 
	\begin{equation}
		\qquad \a^2 D_n^2 \succeq V_n^2 \enspace .
	\end{equation}	
\end{description}

This condition will be checked with \( \mathfrak a \) close to \( 1 \). The 
exact value of \( \a \) depends on \( h \), if \( h \to 0 \) then \( \mathfrak 
a \to 1 \). In the essence, it depends only on the bias between polynomial 
basis and the true density function on the interval.

\subsection{Local Identifiability  Condition}
\begin{description}
	\item[\((\vec{\mathscr L_0})\) \label{condition:l_0}]
	There exists a constant \( \delta_n(\r_0) \), depending on the sample size 
	and concentration raduis such that for each \( \btheta \in \Theta_n(\z) \) 
	it holds:
	\begin{equation}
	\|
		\mathbf I_p - D_n^{-1} D_n^{2} (\btheta) D_n^{-1}
	\| \leq \delta_n(\r_0) \enspace .
	\end{equation}	
\end{description}

This condition relates the matrices \( D_n^2(\btheta) \) and \( 
D_n^2(\btheta^{\ast}) \) in terms of eigenvectors and eigenvalues. It is a 
standard tool for matrix comparison, and we shall see that many matrices that 
encounter in this article, obey the similar law.

\subsection{Exponential Moment Condition}
\begin{description}
\item[\( (\vec E\! \vec {D_0}) \) \label{condition:ed_0}]
Let \( \vec\zeta = V_n^{-1} \nabla L - \E V_n^{-1} \nabla L \).
There exist constants \( \mathfrak g > 0 \) and \( \nu_0 > 0 \) such that for 
\( \forall \vec \gamma \in \R^p \):
\begin{equation}
	\log \E \exp \left(
		\lambda \dfrac{\vec \gamma^{\top} \vec \zeta }{\| \vec \gamma \|}
	\right) \leq \dfrac{\nu_0^2 \lambda^2}{2} \qquad
	\forall \lambda \colon |\lambda| \leq \mathfrak g
\end{equation}
\end{description}

Both \( \nu_0 \) and \( \mathfrak g \) enter final quantile function and 
probability, so it is possible to perform some nontrivial optimization to 
obtain some sharper bounds. This condition can be satisfied with finite \( 
\nu_0 \) and \( \mathfrak g = \infty \) (we don't give proof of this fact, 
though it can be deduced from how we check this condition in section 
\ref{section:check}), but it is better to choose \( \nu_0 \approx \sqrt{p} 
\) and some finite \( \mathfrak g \), depending on the sample size \( n \) and 
bandwidth \( h \).

\subsection{Concentration Condition}

\begin{description}
\item[\( (\vec{\mathcal C}) \) \label{condition:C}]
The concentration radius \( \r_0 \) satisfies the inequality
\begin{equation}
	\r_0 (1 - \delta_n(\r_0)) \geq 2 \zeta(p, \z) \enspace .
\end{equation}
\end{description}

This condition is mainly an implicit rule for defining \( \r_0 \). It is 
implicit because 
the constant \( \delta \) depends on \( \r_0 \), and this inequality can be 
satisfied for large enough \( n \), because \( \delta_n = O((nh^d)^{-1/2}) \).

We shall see that it is possible to choose particular \( \r_0 \) if the 
effective sample 
size is not very small. Otherwise, we should correct the quantile function \( 
\zeta(p, \z) \) which will lead to different probability in concentration 
theorem.

\section{Constants, Level 2}

In order to satisfy these conditions, we need to establish the relationships 
between the objects from section \ref{section:objects}. We are going to 
reformulate the conditions from the section \ref{section:conditions} in terms 
of the basis \( \bPsi \) and true density function \( f(x) \).

\subsection{Small Oscillation Condition}

Let \( f(x) \) be a true density function. There exists a constant \( c_{f,h} 
\) such that:
\begin{equation}
	\left| 1 - \dfrac{f(x)}{f(x_0)} \right| \leq c_{f,h}
	\enspace , \quad 
	\forall |x-x_0| \leq h \enspace .
\end{equation}

It may seem that condition is rather crude, because in the case of polynomial 
basis we are estimating 
not only the value of the function \( f(x_0) \), but also its derivatives, that 
are contained in the vector \( \btheta^{\bullet} \). The correct estimation 
procedure should lead to correct derivatives. But the influence of this 
constant \( c_{f,h} \), as we will 
see later, is not very large. The bias actually is more important, which 
is of order \( O(h^{p}) \) for polynomial basis in one-dimensional case.

\subsection{Small Bias Condition}
\label{section:small_bias}

There exists a vector \( \btheta^{\bullet} \) and a constant \( B_{p,h} \) such 
that \( \forall t \in [-1, 1] \) it holds:
\begin{equation}
	B_{p,h} \geq \exp\left(
		\varphi(x_0 + th) - 
		\bPsi^\top(t) \btheta^{\bullet}
	\right) \enspace ,
\end{equation}
where \( \varphi(x) = \log f(x) \), with \( f(x) \) as a true density function.

In case of one-dimensional polynomial basis \( \bPsi(t) \) if the function \( 
\varphi(x) = \log f(x) \) is smooth enough, the constant \( B_{p,h} \) is of 
order \( 1 + O(h^{p}) 
\) and can be bounded by
\begin{equation}
	\log B_{p,h} \leq \dfrac{h^p}{p!}\max_{x \in U_h(x_0)} 
	\varphi^{(p)}(x)
	\enspace .
\end{equation}
 In \( d \)-dimensional case, in order to make bias of order \( O(h^p) \), we 
 need to take \( {p+d \choose d} - 1 \) elements of the basis, for example in 
 local quadratic fitting for two-dimensional space, \( \bPsi(\mathbf t) = (1, 
 t_1, t_2, t_1^2, t_1t_2, t_2^2) \).

\subsection{Curve Optimization Condition}

This condition is completely defined by the model and can be calculated by the 
statistician. We require that there exists finite constant \( \c_1 \)
such that
\begin{equation}
\label{eq:curve_optimiation_c1}
	\c_1^2 = \sup_{t \in [-1, 1]}
	\bPsi^\top(t) \left[
		\int_{-1}^{1} K(\tau) \bPsi(\tau) \bPsi^\top(\tau) d\tau
	\right]^{-1} \bPsi(t)
	\enspace .
\end{equation}

The constant \( \c_1 \) depends on basis, and is computable. In case of 
one-dimensional polynomial basis and indicator kernel it equals to
\(
	\c_1^2 = p^2 / 2 \enspace .
\)
The reader can check, for example, that in two-dimensional (\( d = 2 \)) 
quadratic case \( \bPsi(\vec t) = (1, t_1, t_2, t_1^2, t_1t_2, t_2^2)\) with 
indicator kernel this constant is well-defined and equals to \( 
13/2 \).

We can introduce another constant
\begin{equation}
	\c_2^2 = \sup_{t \in [-1, 1]}
		K(t)^2
		\bPsi^\top(t) \left[
			\int_{-1}^{1} K(\tau) \bPsi(\tau) \bPsi^\top(\tau) d\tau
		\right]^{-1} \bPsi(t)
		\enspace ,
\end{equation}
where it clearly holds \( \c_2 \leq \c_1 \). In order to choose the ``best 
model'', both constants should be bounded from above as better as possible.

\subsection{Small Bandwidth Condition}
\label{section:small_bandwidth_condition}

Here we define \( \phi_1 \) and \( \phi_2 \), which depend on 
the true density value \( f_0 = f(x_0) \), and also on oscillation and bias 
constants \( c_{f,h} \) and \( B_{p,h} \), but in a given explicit way:
\begin{eqnarray}
	\nonumber
	\phi_1^2 &=& 2 \int_{-1}^{1} K(\tau) d\tau \cdot \Big(
			(1 \pm 1) c_{f,h} \log f_0 \mp c_{f,h} +
			(1 \pm c_{f,h}) \log(1 \pm c_{f,h})
		\Big)
	\enspace , \\
	\phi_2^2 &=& \int_{-1}^{1} K(\tau) d\tau \cdot f_0^3 (c_{f,h} 
	- \log 
	B_{p,h})^2
	\enspace ,
\end{eqnarray}
where the ``\( \pm \)'' sign stands for maximum of the two expressions with 
``\( - \)'' and ``\( + \)'' respectively.
We require that \( \c_1 \phi_1 < \dfrac{\sqrt{5} - 1}{2} \approx 0.618 \) and 
\( \c_1 \phi_2 < 1 \), this condition arises in the proof of theorem 
\ref{theorem:accurate_small_bias}, and also in check of the conditions 
\nameref{condition:l_0}, \nameref{condition:C}, lemmas \ref{check:l0}, 
\ref{check:C}.

When \( h \to 0 \), this condition is fulfilled automatically, but this 
condition can also serve as an approximate strategy for choosing \( \hat h \) 
if we 
know the 
estimates for \( \hat f(x_0) \) and \( \hat f'(x_0) \).

\subsection{Efective Sample Size Condition}

The lower bound on effective sample size is given as lemma \ref{check:C} and 
requires that 
\begin{equation}
	\sqrt{nh^d} \geq f(x_0) \dfrac{4 \c_1 \zeta(p, \z)}{\log 3/2 \sqrt{1-\c_1 
	\phi_1}}
	\enspace .
\end{equation}
However, in low-density regions where \( f(x_0) \approx 0 \) this approach 
becomes inconsistent. Discussion on this issue is also provided after lemma 
\ref{check:C}. When the effective sample size is too small, the results can be 
modified to remain valid, but with lower probabilities and quantile values.

\section{Main Theorems}
\label{section:theorems}

\subsection{Concentration Result}

\begin{theorem}\label{theorem:concentration}
	Let the conditions \nameref{condition:I}, \nameref{condition:l_0}, 
	\nameref{condition:C}, \nameref{condition:ed_0} be 
	satisfied with some constants \( \mathfrak a \), 
	\(\nu_0\), \( \mathfrak g \), \( \r_0(\z) \)
	. Let 
	\begin{eqnarray}
	\Theta_n(\z)
		&=&
	\Big\{\btheta \colon
		\|D_{n}(\btheta - \btheta^\ast)\|
		\leq 
		\r_0(\z)
	\Big\} \enspace ,
	\\
	\label{eq:r_0}
	\r_0(\z)
		&=&
	4 \mathfrak a \cdot \nu_0 
	(\sqrt{p} + \sqrt{2\z})
	,
	\qquad
	\z \leq \mathfrak g^2 / 4 \enspace .
	\end{eqnarray}
Then 
\begin{equation}
	\P\Big(\tilde{\btheta} \not\in \Theta_n(\z)\Big) \leq 2 e^{-\z} + 8.4 
	e^{-\mathfrak g^2/4}
	\enspace .
\end{equation}
	
\end{theorem}

\begin{remark} There is a condition in the theorem that \( \z \leq
\mathfrak g^2 / 4 \). In fact, it is not very restrictive because \( \mathfrak 
g^2 \) is of order \( nh^d \). However, it is also possible to state the 
theorem 
for infinitely large values of \( \z \), using a second version of the quantile 
function for sub-gaussian quadratic forms. The probability measure of the set 
\( \Theta_n(\z) \) will become \( 2 e^{-\z} \).
\end{remark}

\begin{proof}

Let \( \tilde{\btheta} \notin \Theta_n(\r_0) \). Since \( \tilde{\btheta} \) 
maximizes log-likelihood, we have
\begin{equation}
	\L(\tilde{\btheta}) \geq \L(\btheta^\ast)
	\enspace .
\end{equation}
Since \( \L(\btheta) \) is concave in \( \btheta \), there exists a point
\begin{equation}
	\breve{\btheta} = \lambda \tilde{\btheta} + (1 - \lambda) 
	{\btheta}^\ast, 
	\qquad
	\lambda \in [0, 1]
\end{equation}
with the properties
\begin{equation}
	\| D_n (\breve{\btheta} - \btheta^{\ast}) \| = \r_0
	\enspace , \quad
	L(\breve{\btheta}) \geq L(\btheta^{\ast})
	\enspace .
\end{equation}

    It is enough to show that with probability \( 1 - 2 e^{-\z} - 8.4 
        e^{-\mathfrak g^2/4} \) it holds
    \begin{equation}
        \L(\btheta) - \L(\btheta^\ast) < 0, \qquad
        \forall \btheta \not\in \Theta_n(\z)
        \enspace .
    \end{equation}
    Let us represent log-likelihood in the form
    \begin{equation}
    	\L(\btheta) = S^\top \btheta - A(\btheta)
    	\enspace .
    \end{equation}
    Since \( \nabla \E L(\btheta^\ast) = 0 \),
    \begin{equation}
    	\E S = \nabla A(\btheta^\ast)
    	\enspace .
    \end{equation}
    Therefore, for any \( \btheta \) it holds:
    \begin{eqnarray}
	    L(\btheta) - L(\btheta^\ast)
	    &=&
	    S^\top (\btheta - \btheta^\ast) - 
	    [A(\btheta) - A(\btheta^\ast)] \nonumber \\
	    &=& 
	    (S - \E S)^\top (\btheta - \btheta^\ast) - 
	    \\&&
	    [A(\btheta) - A(\btheta^\ast) - (\btheta - \btheta^\ast)^\top 
	    \nabla 
	    A(\btheta^\ast)]
	    \enspace .
    \end{eqnarray}
    Inspect the first summand:
    \begin{equation}
     	(S - \E S)^{\top} (\btheta - \btheta^{\ast})
		=
     	[D_n^{-1} (S - \E S)]^{\top} D_n (\btheta - \btheta^{\ast})
    \end{equation}
    For vector \( \bxi = D_n^{-1} \nabla L(\btheta^{\ast}) = D_n^{-1} (S - \E 
    S) \) it follows by lemma \ref{lemma:quadratic} for \( \a \) and 
    \( \nu_0 \) from conditions section \ref{section:conditions} that:
    \begin{equation}
    \P
    \big(
    \| \bxi \| \geq \zeta(p, \z)
    \big)
    \leq 2 e^{-\z} + 8.4 e^{-\mathfrak g^2/4}, \qquad
    \zeta^2 (p, \z) = \a^2 \nu_0^2 (p + 2\sqrt{2p\z} + 2 \z)
    \enspace .
    \end{equation}
    Therefore,
    \begin{equation}
    	\| S^\top (\btheta - \btheta^\ast) \| = \| \bxi^\top D_n(\btheta - 
    	\btheta^\ast) \| \leq \zeta(p, \z) \cdot \r_0
    	\enspace .
    \end{equation}
    The second summand, by Taylor expansion, can be represened as
    \begin{equation}
    	A(\btheta) - A(\btheta^\ast) - (\btheta - \btheta^\ast)^\top 
    	\nabla 
    	A(\btheta^\ast) = \frac12 (\btheta - \btheta^{\ast})^\top 
    	\nabla^2 
    	A(\bar{\btheta}) (\btheta - \btheta^\ast)
    	\enspace .
    \end{equation}
    By condition \( (\mathcal L_0) \) with \( \delta_n(\r_0) \) and \( \bar 
    \btheta \in \Theta_n(\z) \) it follows that
    \begin{equation}
    	\frac12 (\btheta - \btheta^{\ast})^\top \nabla^2 
    	A(\bar{\btheta}) (\btheta - \btheta^\ast) \geq \frac{1 - 
    	\delta_n(\r_0)}{2} \| D_n (\btheta - \btheta^\ast) \|^2
    \enspace .
    \end{equation}
    Thus, with probatility at least \( 1 - 2e^{-\z} - 8.4 e^{-\mathfrak 
        g^2/4} \) it follows that
    \begin{equation}
    	L(\breve{\btheta}) - L(\btheta^\ast) \leq \zeta(p, \z) \r_0 - 
    	\frac{1 
    	- \delta_n(\r_0)}{2} \r_0^2 \leq 0
    \enspace ,
    \end{equation}
    which is contradiction, according to the condition 
    \nameref{condition:C},
    \begin{equation}
    	\r_0(\z) (1 - \delta_n(\r_0(\z))) \geq 2 \zeta(p, \z)
    	\enspace .
    \end{equation}
    
    \textbf{End of the proof of theorem \ref{theorem:concentration}.}
\end{proof}

\subsection{Fisher Theorem}
This theorem describes finite-sample approximation of the distribution of the 
estimate $\tilde{\btheta}$ in terms of $D_n$ and score vector \( \bxi \).
\begin{theorem}\label{theorem:fisher}
Let the conditions \nameref{condition:C}, \nameref{condition:I}, 
\nameref{condition:ed_0}, \nameref{condition:l_0} hold.

Then for \( \tilde{\btheta} \in \Theta_n(\z) \) from theorem 
\ref{theorem:concentration} with dominating probability
\begin{equation}
	\P\left(
		\tilde \btheta \in \Theta_n(\z)
	\right) \geq 1 - (2 e^{-\z} + 8.4 e^{-\mathfrak g^{2}/4})
\end{equation}
it holds

\begin{equation}
    \| D_n (\tilde{\btheta} - \btheta^\ast) - \bxi \| \leq \diamondsuit(n, 
    \z)
    \enspace , \quad \z \leq \mathfrak g^2/4 \enspace ,
\end{equation}
where $\diamondsuit(n, \z) = \r_0(\z) \cdot \delta_n(\r_0)$, and $\r_0$ is 
defined 
by (\ref{eq:r_0}).
\end{theorem}

\begin{remark}
The vector \( (nh)^{-1/2}\bxi_n \) is asymptotically standard normal. 
Following the classical statistics, the difference between the centered 
parameter and the standard normal random variable is of order \( 
(nh^d)^{-1/2} \):
\begin{equation}
	\| d_0 (\tilde \btheta - \btheta^{\ast}) - \mathcal N(0,1 ) \| = 
	O((nh^d)^{-1/2}) \enspace .
\end{equation}
The Fisher theorem is the asymptotic refinement of the Central Limit Theorem.
Indeed, asymptotic behavior of the term $\delta_n(\r_0)$ is the following: 
with $nh 
    \to \infty$, we have $\delta_n(\r_0) \to 0$, $\diamondsuit(n,\z) = 
    O((nh)^{-1/2})$.
    
While $\tilde{\btheta} - \btheta^\ast = O((nh)^{-1/2})$, the Fisher 
    result can be written in the form:
    \begin{equation}
        \left\|
            d_0 (\tilde{\btheta} - \btheta^\ast)
            - (nh^d)^{-1/2} \bxi_n
        \right\|
        = O((nh^d)^{-1}) \enspace .
    \end{equation}
\end{remark}

\begin{proof}
    The principle step is a bound on the local linear approximation of the 
    stochastic part of the gradient $\nabla \L(\btheta)$.    
    Although $\bxi$ is random, it can be shown that $\bxi$ depends only on 
    $\tilde{\btheta}$.
    
    Indeed, since $\nabla \L(\tilde{\btheta)}$ is zero, and stochastic part of 
    $\L$ is linear on $\btheta$,
    \begin{equation}
        \bxi(\btheta^\ast) = D_n^{-1} \nabla \L (\btheta^\ast)
        = D_n^{-1}
            [\nabla \L(\btheta^\ast)
            - \nabla L(\tilde{\btheta})]
        = D_n^{-1}
            [\nabla \E \L(\btheta^\ast)
            - \nabla \E \L(\btheta) |_{\btheta = \tilde\btheta}] \enspace .
    \end{equation}

    
    Next, we can bound the norm of  \( (D_n(\tilde \btheta - \btheta^{\ast}) - 
    \bxi) \) by multiplying it by an arbitrary vector \( \vec u \) of unit norm:
    \begin{eqnarray*}
    	\vec u^\top [D_n(\tilde \btheta - \btheta^{\ast}) - \bxi]
    	& = & 
    	\vec u^{\top} \left[
    		D_n (\tilde \btheta -  \btheta^{\ast}) - D_n^{-1}
    		(\nabla \E L(\btheta^{\ast}) - \nabla \E L(\tilde \btheta))
    	\right] \\
    	& = &
    	\vec u^{\top} \left[
    		D_n (\tilde \btheta -  \btheta^{\ast}) - D_n^{-1} D_n^{2} (\bar{
    		\btheta}_{\vec u}) (\tilde \btheta - \btheta^{\ast})
    	\right]	
    	\\
    	& = &
    	\vec u^{\top} \left[
    			\mathbf I_p
    	    	- D_n^{-1} D_n^{2} (\bar \btheta_{\vec u}) D_n^{-1}
    	\right] (\tilde \btheta - \btheta^{\ast}) \\
    	& \leq &
    	\| \vec u^\top (\mathbf I_p - 
    	    	    	D_n^{-1} D_n^{2} (\bar 
    	    	    	\btheta_{\vec u}) D_n^{-1}) \|
    	\cdot \|
    		D_n (\tilde \btheta - \btheta^{\ast})
    	\| \enspace ,
    \end{eqnarray*}
    where $\bar{\btheta}_{\vec u} = \lambda \tilde{\btheta} + (1 - \lambda) 
        \btheta^\ast$, $\lambda \in [0, 1]$.
        
    By theorem \ref{theorem:concentration}, with high probability 
    it holds 
    $\tilde{\btheta} \in \Theta_n(\z)$.
    By condition \nameref{condition:l_0}, for each \( \bar{\btheta} \in 
    \Theta_n(\z) \) it holds \(
            \big\|
                \mathbf I_p - D_n^{-1}D_n^2(\bar{\btheta})D_n^{-1}
            \big\| \leq \delta(\r_0)
        \), 
    so we have
        \begin{equation}
            \| D_n(\tilde{\btheta} - \btheta^\ast) - \bxi \| \leq \r_0(\z) 
            \delta(\r_0) \enspace .
        \end{equation}

    \textbf{End of the proof of theorem \ref{theorem:fisher}.}
\end{proof}

\subsection{Wilks Theorem}

\begin{theorem}[Spokoiny, \cite{Spokoiny}]\label{theorem:wilks}
Let the conditions \nameref{condition:C}, \nameref{condition:I}, 
\nameref{condition:ed_0}, \nameref{condition:l_0} hold.

Then for \( \tilde{\btheta} \in \Theta_n(\z) \) from theorem 
\ref{theorem:concentration} with dominating probability
\begin{equation}
	\P\left(
		\tilde \btheta \in \Theta_n(\z)
	\right) \geq 1 - (2 e^{-\z} + 8.4 e^{-\mathfrak g^{2}/4})
\end{equation}
with \( \diamondsuit(n, \z) \) from theorem \ref{theorem:fisher} it holds:
\begin{eqnarray}
    \left|
        \sqrt{2\L(\tilde{\btheta}, \btheta^\ast)} - \| D_n(\tilde{ \btheta} - 
        \btheta^\circ) \|
    \right| &\leq& 2 \diamondsuit(n, \z)
    \\
    \left|
        \sqrt{2\L(\tilde{\btheta}, \btheta^\ast)} - \| \bxi \|
    \right| &\leq& 3 \diamondsuit(n, \z)
\end{eqnarray}
\end{theorem}

\begin{remark}
The constant $\diamondsuit(n, \z)$ is familiar from the theorem 
\ref{theorem:fisher}, the theorem is an asymptotic refinement to behavior of 
likelihood when $nh^d \to \infty$.

    The proof can be found in Spokoiny \cite{Spokoiny}. The theorem is valid 
    under conditions, formulated in the abovementioned article, which are 
    checked in the current text.
 \qed
\end{remark}

\subsection{Accurate Small Bias Result}

\begin{theorem}
\label{theorem:accurate_small_bias}
Suppose that \( \c_1 \phi_1 \leq \frac{\sqrt{5}-1}{2} \) and \( \c_2 \phi_2 
\leq 1 \),
\(
		I_k
		=
		\int_{-1}^{1} K(t) dt \leq 2^d
\),
\begin{equation}
	\varepsilon = \max\{
		\c_1 \phi_1 (1 - \c_1 \phi_1)^{-1/2}, \c_1 \phi_2
	\} 
	\enspace .
\end{equation}
Then it holds:
	\begin{equation}
		\| d_0(\btheta^{\circ}) (\btheta^{\ast} - \btheta^{\bullet}) \|
		\lesssim
		 \sqrt{p} \sqrt{I_K}
		 (1-\eps)^{-1}
		 (1 + c_{f,h})
		 \cdot
		 f(x_0) 
		 \cdot
		 |B_{p,h} - 1|
		\enspace .
	\end{equation}
\end{theorem}

\begin{remark}
There are results of a kind \( \btheta^{\ast} \approx \btheta^{\circ} \) and \( 
\btheta^{\circ} \approx \btheta^{\bullet} \), in terms of curvature matrix, see 
lemmas \ref{lemma:smb} and \ref{lemma:constant_approximation}. However, we 
cannot combine these results to obtain the final bound, because it is an 
asymptotic refinement of order \( O(h^p) \) instead of \( O(h) \). More 
precisely, the term \( 
|B_{p,h} - 1| \) is of order \( O(h^p) \), other terms are of order \( 1 + O(h) 
\), and the following approximate inequality holds:
\begin{equation}
	\| d_0(\btheta^{\circ}) (\btheta^{\ast} - \btheta^{\bullet}) \|
	\lesssim
	\sqrt{2} f(x_0) |B_{p,h} - 1|
	\enspace .
\end{equation}
\end{remark} 

\begin{proof}
From the conditions \( \c_1 \phi_1 < \dfrac{\sqrt 5 - 1}{2} \) it follows 
that \( \c_1 \phi_1 (1 - \c_1 \phi_1)^{-1/2} < 1 \). We will use this 
observation later. Combined with \( \c_1 \phi_2 < 1 \) this allows to claim 
that \( \varepsilon < 1 \).

Let \( \btheta_1, \btheta_2 \in [\btheta^{\ast}, \btheta^{\bullet}] \). If 
there is a point \( \btheta^t = t \btheta^{\bullet} + (1 - t) \btheta^{\ast} 
\), \( t \in [0,1] \), then there is a representation
\begin{equation}
	\btheta^t - \btheta^{\circ} = t \btheta^{\bullet} + (1 - t) \btheta^{\ast}  
	- t \btheta^{\circ} - (1-t) \btheta^{\circ} = 
	t (\btheta^{\bullet} - \btheta^\circ) + (1-t) (\btheta^{\ast} - 
	\btheta^{\circ}) \enspace .
\end{equation}
Hence, we have a triangle inequality in terms of curvature matrix \( 
d_0(\btheta^{\circ}) \):
\begin{eqnarray}	
	\nonumber
	\| d_0(\btheta^{\circ}) (\btheta^{t} - \btheta^{\circ}) \|
	& \leq &
	t \| d_0(\btheta^{\circ}) (\btheta^{\bullet} - \btheta^\circ) \|
	+ (1-t) \| d_0(\btheta^{\circ}) (\btheta^{\ast} - 
		\btheta^{\circ}) \|
	\\
	& \leq &
	\max \left\{
			\| d_0(\btheta^{\circ}) (\btheta^{\ast} - \btheta^{\circ}) \|, 
			\| d_0(\btheta^{\circ}) (\btheta^{\bullet} - \btheta^{\circ}) \|
		\right\} \enspace .
\end{eqnarray}
Therefore, according to lemmas \ref{lemma:smb} and 
\ref{lemma:constant_approximation}, at each of the points \( \btheta^t \in 
\{\btheta_1, \btheta_2\} \) it holds
\begin{eqnarray}
	\nonumber
	\| \mathbf I_p - D(\btheta^{t}) D^{-2} (\btheta^{\circ}) D(\btheta^{t}) \| 
	& \leq &
	\exp\Big( \bPsi^\top(t) d_0^{-1} d_0 (\btheta^{t} - \btheta^{\circ}) 
	\Big) - 1 \\
	\nonumber
	& \lesssim &
	\| \bPsi^\top(t) d_0^{-1} \| \cdot \| d_0 (\btheta^{t} - \btheta^{\circ}) \|
	\\ \nonumber
	& \leq &
	\c_1 \cdot \max \left\{
		\| d_0(\btheta^{\circ}) (\btheta^{\ast} - \btheta^{\circ}) \|, 
		\| d_0(\btheta^{\circ}) (\btheta^{\bullet} - \btheta^{\circ}) \|
	\right\}
	\\ \nonumber
	& \leq &
	\c_1 \cdot \max \left\{
		\phi_1 (1 - \c_1 \phi_1)^{-1/2}, 
		\phi_2
	\right\}
	\enspace .
\end{eqnarray}
Since \( \eps = \c_1 \max \left\{
	\phi_1  (1 - \c_1 \phi_1)^{-1/2}, \phi_2
\right\} \), we obtain:
\begin{eqnarray}
	\| \mathbf I_p - D(\btheta^{t}) D^{-2} (\btheta^{\circ}) D(\btheta^{t}) \| 
		& \leq & \eps
		\enspace .
\end{eqnarray}
Next, we construct a bound using two Taylor expansions. Let \( g(\btheta) = 
(nh)^{-1} 
\E L(\btheta) \).
\begin{eqnarray}
g(\btheta^{\ast})
=
g(\btheta^{\bullet}) + \nabla g(\btheta^{\bullet})^\top 
(\btheta^{\ast} - \btheta^{\bullet}) + \dfrac12 (\btheta^{\ast} - 
\btheta^{\bullet})^\top d_0^2(\btheta_1) (\btheta^{\ast} - 
\btheta^{\bullet}) \\
g(\btheta^{\bullet})
=
g(\btheta^{\ast}) + \nabla g(\btheta^{\ast})^\top 
(\btheta^{\bullet} - \btheta^{\ast}) + \dfrac12 (\btheta^{\ast} - 
\btheta^{\bullet})^\top d_0^2(\btheta_2) (\btheta^{\ast} - 
\btheta^{\bullet})	
\end{eqnarray}

Adding the two expressions, we obtain
\begin{eqnarray}	
	\nonumber
	(1 - \eps)
	\big\| d_0 (\btheta^\circ) (\btheta^{\ast} - 
	\btheta^{\bullet}) \big\|^2
		&\leq&
	(\btheta^{\ast} - \btheta^{\bullet})^\top \dfrac{d_0^2(\btheta_1) + 
	d_0^2 
	(\btheta_2)}{2} (\btheta^{\ast} - \btheta^{\bullet})\\
		&\leq&
	\| \nabla \big( g(\btheta^{\bullet}) - g(\btheta^{\ast}) \big)^\top 
	(\btheta^{\ast} - \btheta^{\bullet}) \|
\end{eqnarray}

Then latter, by Cauchy inequality, can be bounded by
\begin{eqnarray}	
	\nonumber
	\| \nabla \big( g(\btheta^{\bullet}) - g(\btheta^{\circ}) \big)^\top 
	(\btheta^{\ast} - \btheta^{\bullet}) \| 
	& \leq &
	\| d_0 (\btheta^{\circ}) (\btheta^{\ast} - \btheta^{\circ}) \|
	\times
	\\
	&&
	\| d_0 (\btheta^{\circ})^{-1} (\nabla g(\btheta^{\ast}) - \nabla 
	g(\btheta^{\bullet})) \| \enspace .
\end{eqnarray}

Therefore, after cancelling \( \|d_0(\btheta^{\circ}) (\btheta^{\ast} - 
\btheta^{\bullet}) \| \) from both sides, we have, according to the lemma 
\ref{lemma:gradient_difference}:
\begin{eqnarray}
	\nonumber
	\| d_0(\btheta^{\circ})(\btheta^{\ast} - 
	\btheta^{\bullet}) \|
	& \leq &
	(1 - \eps)^{-1}
	\| d_0^{-1} (\btheta^{\circ}) (\nabla g(\btheta^{\ast}) - \nabla 
	g(\btheta^{\bullet})) \|\\
	& \leq &
	\label{eq:bias_bound}
	(1-\eps)^{-1}
	|1 - B_{p,h}| \cdot \sqrt{p h^{-d} \mathrm{pr}_2(x_0)}
	\enspace ,
\end{eqnarray}
where
\begin{equation}
	\sqrt{h^{-d} \mathrm{pr}_2 (x_0)} \leq 
	\sqrt{\int_{-1}^{1} K(t) f(x_0)^{2} (1 + c_{f,h})^2 dt}
	\leq
	\sqrt{I_K} f(x_0) (1 + c_{f,h})
	\enspace .
\end{equation}
\textbf{End of the proof of theorem \ref{theorem:accurate_small_bias}.}
\end{proof}

\section{Checking Conditions}
\label{section:check}

\subsection{Identification Condition \nameref{condition:I}}

\begin{lemma}
\label{check:I} 
Let the conditions from theorem \ref{theorem:accurate_small_bias} hold. Then
the constant \( \mathfrak a \) in condition \nameref{condition:I}:
	\begin{equation}
		\a^2 D_n^2 \succeq V_n^2
	\end{equation}
	can be 
	bounded above by
	\begin{eqnarray}
		\mathfrak a^2
		& \leq &
			\sup_{|x - x_0| \leq h} 
			\dfrac{K((x-x_0)/h) f(x)}
			{\exp(\bPsi^\top \btheta^{\ast})} \\
		& \leq &
		B_{p,h} \exp( \c_1 \cdot I_K^{1/2} (1 - \varepsilon)^{-1} 
					(1+c_{f,h}) f(x_0) |B_{p,h} - 1| ) \enspace ,
	\end{eqnarray}
	where \( I_K = \int_{-1}^{1} K(t) dt \),
\(
	\varepsilon = \max\{
		\c_1 \phi_1 (1 - \c_1 \phi_1)^{-1/2}, \c_1 \phi_2
	\} 
	\enspace .
\)
\end{lemma}

\begin{remark}
	In case of one-dimensional local polynomial estimation with indicator 
	kernel, the quantity \( |B_{p,h} - 1| \) is or order \( O(h^p) \), the 
	multiples \( (1 - \varepsilon) \), \( (1 + c_{f,h}) \) are or order \( 1 
	\). Therefore, for \( h \) small enough, \( h < 1 \), we have
	\begin{equation}
		\a^2 \lesssim B_{p,h} (1 + \sqrt{p I_K} \c_1 f(x_0)\cdot O(h^p)) 
		\approx 1 + \sqrt{p I_K} \c_1 
		f(x_0) \cdot |B_{h,p} - 1| \enspace .
	\end{equation}
\end{remark}

\begin{proof}
Firstly bound $V_n^2 = \Var \nabla \L$.
\begin{eqnarray*}
    \frac{1}{n} (V_n^2)_{ij} &=& \int_\X K^2 \Psi_{i}\Psi_j f(x) dx
    - \int_\X K \Psi_{i} f(x) dx
    \int_\X K \Psi_{j} f(x) dx \\
    &\preceq&
        \int_\X K^2 \Psi_{i}\Psi_j f(x) dx.
\end{eqnarray*}
We use the fact that the second summand is minus non-negative definite matrix 
with rank 1. Since both matrices $D_n^2$ and first summand of $V_n^2$ can be 
represented in the 
form 
suitable for lemma \ref{lemma:eigenvalues}, we can apply it and bound the 
maximal eigenvalue of $D_n^{-1}V_n^2 D_n^{-1}$. 

Recall that
\begin{equation}
    D_n^2 = \int_{\X} K \bPsi \bPsi^\top \exp(\bPsi^\top(x) \btheta^\ast) dx 
    \enspace .
\end{equation}

In terms of lemma \ref{lemma:eigenvalues} their diagonal operators are, 
correspondingly,\\ $K(t) \exp(\bPsi^\top (t) \btheta^\ast) $ and $K^2(t) 
f(x)$. So, $\a^2$ can be estimated with
\begin{equation}\label{eq:wrong_bound_for_a_2}
    \a^2 \leq \sup_{x \in U_h(x_0)} \frac{K(t) f(x)}{\exp(\bPsi(t)^\top 
    \btheta^\ast)}
    \enspace .
\end{equation}
Then we use trivial bound \( K(t) \leq 1 \). It is possible to write
\begin{equation}
	\sup_{|x_0 - x| \leq h} \dfrac{f(x)}{\exp(\bPsi^\top 
	\btheta^\ast)} \leq \exp(\bPsi^\top(t) (\btheta^\bullet - \btheta^\ast) 
	+ 
	\log B_{p,h}) \enspace .
\end{equation}
Using the result of theorem \ref{theorem:accurate_small_bias} we obtain
\begin{eqnarray}
	\a^2
	& \leq &
		B_{p,h} \exp( \bPsi^\top (t) d_0^{-1}(\btheta^{\circ}) 
		d_0(\btheta^{\circ})
		(\btheta^{\bullet} - \btheta^{\ast}) )\\
	& \leq &
		B_{p,h} \exp( \c_1 \cdot \sqrt{p}\sqrt{I_K} (1 - \varepsilon)^{-1} 
		(1+c_{f,h}) f(x_0) |B_{p,h} - 1| ) \enspace .
\end{eqnarray}
\textbf{End of the proof of lemma \ref{check:I}.}
\end{proof}

\subsection{Local Identifiability Condition \nameref{condition:l_0}}

\begin{lemma}
\label{check:l0}
	For all \( \btheta \in \Theta_n(\z) \)
	local identifiability condition
	\begin{equation}
		\| \mathbf I_p - D^{-1} D^2(\btheta) D^{-1} \| \leq \delta_n(\r_0(\z))
	\end{equation}
	holds with the constant
    \begin{equation}
    	\delta_n(\r_0) \leq
    	\exp\left(
			\dfrac{\c_1 \r_0}{\sqrt{1 - c_1 \phi_1} \sqrt{f(x_0)nh^d}}
    	\right) - 1 \enspace .
    \end{equation}
\end{lemma}

\begin{remark}
	When effective sample size is large, and \( h \) is small, the above 
	expression is equivalent to \( \delta_n \lesssim \dfrac{\c_1 
	\r_0}{\sqrt{f(x_0)nh^d}} \). We will see later that \( r_0 \) can be chosen 
	as \( 4 \zeta(p, \z) \).
\end{remark}
\begin{proof}
    Maximal absolute eigenvalue of ($\mathbf I - X$) is equal to \( 
    \max(|\lambda_{\min}(X) - 1|, |\lambda_{\max}(X) - 1|) \). From lemma 
    \ref{lemma:eigenvalues} it follows that $\lambda(D_n^{-1} D_n^2(\btheta) 
    D_n^{-1})$ belongs to the interval
    \begin{equation}
        \left[
        \min_{x \in U_h(x_0)}
            \exp\big(\bPsi(x)^\top (\btheta - \btheta^\ast)\big),
        \max_{x \in U_h(x_0)}
            \exp\big(\bPsi(x)^\top (\btheta - \btheta^\ast)\big)
        \right]
        \enspace .
    \end{equation}
    Let \( \vec v = \btheta - \btheta^{\ast} \).
    Note that \( \bPsi^\top \vec v = \bPsi^{\top} D_n^{-1}(\btheta^{\ast}) 
    D_n(\btheta^{\ast}) 
    \vec v \) and consequently, the matrices \( D_n^2(\btheta^{\ast}) \) and 
    \( D_n^2(\btheta^{\circ}) \) are related through lemma \ref{lemma:smb}.
    Therefore, 
    \begin{equation}
    	\| \bPsi^\top D_n^{-1}(\btheta^{\ast}) \|^2 \leq (f(x_0) nh^d)^{-1} 
    	\c_1^2
    	\cdot (1 - \c_1 \phi_1)^{-1} \enspace , \quad
    	\| D_n(\btheta^{\ast}) \vec v \| \leq \r_0 \enspace ,
    \end{equation}
    and by Cauchy inequality the constant \( \delta_n \) is bounded by
    \begin{equation}
    	\delta_n(\r_0) \leq
    	\exp\left(
			\dfrac{\c_1 \r_0}{\sqrt{1 - \c_1 \phi_1} \sqrt{f(x_0)nh^d}}
    	\right) - 1 \enspace .
    \end{equation}
    \textbf{End of the proof of lemma \ref{check:l0}.}
\end{proof}

\subsection{Concentration Condition \nameref{condition:C}}

\begin{lemma}
Under condition 
\begin{equation}
	\sqrt{nh^d} \geq f(x_0)
	\dfrac
	{4 \c_1 \zeta(p,z)}
	{\log 3/2 \sqrt{1 - \c_1 \phi_1}}
	\enspace ,
\end{equation}
the concentration condition 
\nameref{condition:C} holds:
\begin{equation}
	\r_0(\z) (1 - \delta_n(\r_0)) \geq 2 \zeta(p, \z) \enspace .
\end{equation}
\label{check:C}
\end{lemma}
\begin{proof}
Note that \( \delta_n(\r_0) \to 0 \) when \( nh^d \to 0 \). We will need \( 
nh^d > N_0 \) such that  \( \delta_n(\r_0) \leq 1/2 \). This will allow us to 
take \( \r_0(\z) = 4 \zeta(p, \z) \). The condition turns into
\begin{equation}
	\exp\left(\dfrac{\c_1 \r_0}{\sqrt{1 - \c_1 \phi_1}\sqrt{f(x_0)nh^d}}\right) 
	\leq 3/2
	\enspace ,
\end{equation}
which turns into
\begin{equation}
	\sqrt{nh^d} \geq (1 - \c_1 \phi_1)^{-1/2} f(x_0) 
	\dfrac{4 \c_1 \zeta(p, \z)}{\log 3/2}
\end{equation}

\textbf{End of the proof of lemma \ref{check:C}.}
\end{proof}

\begin{remark}
When the density is small, there will be no concentration and the sample size 
will be too small for this condition. We can redefine \( \zeta_n(p, \z) 
\) as maximal value that satisfies the concentration condition
\begin{equation}
	\zeta_n(p, \z) = \sqrt{1 - \c_1 \phi_1}\dfrac{\sqrt{nh^d} \log 3/2}{4 \c_1 
	f(x_0)} \enspace ,
	\quad
	\r_n(\z) = 4 \zeta_n(p, \z) \enspace ,
\end{equation}
and the probability in theorems \ref{theorem:concentration} and 
\ref{theorem:fisher} becomes
\begin{equation}
	\P(
		\tilde \btheta \notin \Theta_n(\z)
	) =
	\P (
		\| \bxi \| > \zeta_n(p, \z)
	) \enspace .
\end{equation}
Finally, for small \( \sqrt{nh^d} \) the concentration theorem loses its 
``concentration'' and can be stated in
the following form:
\begin{equation}
	\| d_0(\btheta^{\ast}) (\tilde \btheta - \btheta^{\ast}) \| \leq
	\sqrt{1 - \c_1 \phi_1}\dfrac{\log 3/2}{4 \c_1 
		f(x_0)} \enspace ,
\end{equation}
where right-hand side is no longer of order \( (nh^d)^{-1/2} \).

\end{remark}

\subsection{Exponential Moment Condition \nameref{condition:ed_0}}

\begin{lemma}
\label{check:ed0}
Let \( \vec \zeta = \nabla L - \E \nabla L \). Consider the function
\begin{equation}
	M(\lambda, \vec\gamma) = \log \E \exp \left(
		\lambda \vec \gamma^\top V_n^{-1} \vec\zeta
	\right) 
\end{equation}

For all \( |\lambda| \leq \mathfrak g \) the following inequalty holds:
\begin{equation}
	M(\lambda, \vec \gamma) \leq \dfrac{\nu_0^2 \lambda^2}{2} \enspace , \quad
	\nu_0^2 = p + \dfrac{16 \mathfrak g C_{V,f}^3}{\sqrt{nh^{3d}}} \enspace ,
\end{equation}
where \( C_{V,f} \) is defined in lemma \ref{lemma:relation_of_core_constants}, 
and satisfies 
\begin{equation}
	C_{V,f}^{2} \leq (1 - c_{f,h})^{-1} f_0^{-1} \c_2^2 + \dfrac{h^d}{1 - 
	\mathrm{pr}_1(x_0)} \enspace , \quad
	\mathrm{pr}_1(x_0) = \int_{-1}^{1} K((x-x_0)/h)f(x) dx
	\enspace .
\end{equation}

\end{lemma}

\begin{proof}
Since \( \vec \zeta \) can be represented as a sum of i.i.d. random variables 
\begin{equation}
	\vec \zeta = \sum_{i=1}^{n} \vec \zeta_i
	= \sum_{i=1}^{n} (K_i \bPsi_i - \E K_i \bPsi_i)
	\enspace ,
\end{equation}
 the function \( M(\lambda, \vec \gamma) 
\) can be also rewritten as
\begin{equation}
	M(\lambda, \vec \gamma) =
	n \log \E \exp(\lambda \vec \gamma^\top V_n^{-1} \vec\zeta_1) \enspace .
\end{equation}
Consider Taylor expansion of degree 3 at \( \lambda = 0 \) for \( M(\lambda, 
\vec\gamma) 
\): there exists \( \overline{\lambda} \in [0, \lambda] \) such that
\begin{equation}
	M(\lambda, \vec\gamma)
	=
	M(0) + \lambda M'(0, \vec \gamma) + \dfrac{\lambda^2}{2} M''(0, \vec 
	\gamma) 
	+ 
	\dfrac{\lambda^3}{6} 
	M'''(\overline{\lambda}, \vec \gamma)
\end{equation}
Denote \( u = \vec\gamma^\top V_n^{-1} \vec\zeta \), \( u_1 = \vec\gamma^\top 
V_n^{-1} \vec\zeta_1 \) for brevity.
Careful differentiation gives us that
\begin{equation}
	M'(\lambda) = \dfrac{\E(u \exp(\lambda u))}{\E \exp(\lambda u)},
	\ 
	M''(\lambda) =
	\dfrac
	{
		\E (u^2 \exp(\lambda u))\E \exp(\lambda u)
		-
		(\E u \exp(\lambda u))^2
	}
	{\left( \E \exp(\lambda u)\right)^2}
	\enspace ,
\end{equation}
\begin{equation}
	\dfrac 1 n
	M'''(\lambda) =
	\dfrac
		{\E (u_1^3 \exp(\lambda u_1))}
		{\E \exp(\lambda u_1)}
	- \dfrac
		{3 \E (u_1^2 \exp(\lambda u_1))\E (u \exp(\lambda u_1))}
		{(\E \exp \lambda u_1)^2}
	+ \dfrac 
		{2 (\E u_1 \exp(\lambda u_1))^3}
		{(\E \exp(\lambda u_1))^3}
	\enspace ,
\end{equation}
and after substituting \( \lambda = 0 \) we obtain \( M(0) = M'(0) = 0 \),
\( M''(0) = \E u^2 \). Then
\begin{equation}
	M''(0) \leq \E \sup_{\| \vec\gamma \| = 1} \left( \vec \gamma^\top V_n^{-1} 
	\vec\zeta \right)^2 = \E \vec \zeta^\top V_n^{-2} \vec \zeta = \mathrm{Tr}\ 
	\E\vec 
	\zeta \vec \zeta^\top V_n^{-2} = p \enspace .
\end{equation}
Let us show that \( |u_1|^2 \leq \frac{4 C_{V, f}^2}{nh^d} \), where \( C_{V, 
f} \) is defined in lemma \ref{lemma:relation_of_core_constants}. First, we 
will show that the square of uncentered random variable \( \vec \gamma^\top 
V_n^{-1} K(T_1) \bPsi(T_1) \) is bounded:
\begin{equation}
	|\vec \gamma^\top V_n^{-1} K(T_1) \bPsi(T_1)|^2 \leq K(t)^2 \bPsi(t)^\top 
	V_n^{-2} \bPsi(t) \leq \dfrac{1}{nh^d} C_{V, f}^2 \enspace .
\end{equation}
Then, the centered variable is naturally bounded by twice the bound of 
noncentered, therefore the bound for square multiplies \( 4 \) times.

This observation allows to obtain the bound for \( M'''(\lambda) \), using the 
fact that if random variable \( X \) is bounded by \( |X| \leq c \), then \( \E 
XY \leq c \cdot \E |Y| \):
\begin{equation}
	\dfrac{1}{n} M'''(\lambda) \leq 6 \cdot
	\left(\dfrac{2 C_{V, f}}{\sqrt{nh^d}} \right)^3
	\enspace .
\end{equation}
Thus, we obtained the bound for the whole expression \( M(\lambda, \vec \gamma) 
\):
\begin{equation}
	M(\lambda, \vec \gamma) \leq \dfrac{\lambda^2}{2} p + \lambda^3 \cdot
	\dfrac{8C_{V,f}^3 n}{(\sqrt{nh^d})^3} = \dfrac{\lambda^2}{2} \left(
		p + 2 \lambda \dfrac{8C_{V,f}^3}{\sqrt{nh^{3d}}}
	\right) \leq
	\dfrac{\lambda^2}{2} \left(
		p + \dfrac{16 \mathfrak g C_{V,f}^3}{\sqrt{nh^{3d}}}
	\right)
	\enspace .
\end{equation}
\textbf{End of the proof of lemma \ref{check:ed0}.}
\end{proof}

\appendix

\section{Technical Results}

\subsection{Small Bias Results for $ \btheta^{\circ}, \btheta^{\bullet}, 
\btheta^{\ast} $.}

\begin{lemma}
\label{lemma:smb}
Let \( \c_1 \phi_1 < 1 \), \( f_0 = f(x_0) \), where \( \c_1 \) is given by 
(\ref{eq:curve_optimiation_c1}),
\begin{eqnarray}
	\phi_1^2 &=& 2 I_k f_0 I_K \Big(
			(1 \pm 1) c_{f,h} \log f_0 \mp c_{f,h} +
			(1 \pm c_{f,h}) \log(1 \pm c_{f,h})
		\Big) \enspace ,\\
	I_k &=& \int_{-1}^{1} K(t) dt \leq 2^d \enspace ,
\end{eqnarray}
where the ``\( \pm \)'' sign stands for the maximum of two expressions with 
plus and minus.
Then the following holds:
\begin{eqnarray}
	\| \mathbf I_p - D_n(\btheta^{\ast}) D_n^{-2} (\btheta^{\circ}) 
	D_n(\btheta^{\ast})  \| 
	& \lesssim & \c_1 \phi_1 (1 - \c_1 \phi_1)^{-1} \enspace ,\\
	\| d_0 (\btheta^{\circ}) (\btheta^{\circ} - \btheta^{\ast}) \|^2
	& \lesssim & f_0 \phi_1^2 (1 - \c_1 \phi_1)^{-1} \enspace .
\end{eqnarray}
\end{lemma}

\begin{remark}
	The quantity \( \phi_1^2 \) is proportional to $|2c_{f,h} \log f_0|$ 
	which 
	is of order \( O(h) \), so with \( h \to 0 \) it 
	holds \( \| \mathbf I_p - 
	D(\btheta^{\ast}) D^{-2} (\btheta^{\circ}) 
	D(\btheta^{\ast})  \|  = O(h^{1/2}) \), \( \btheta^{\circ} - \btheta^{\ast} 
	= O(h) \).
\end{remark}

\begin{proof}
Consider the expectation of log-likelihood under the true measure:
\begin{eqnarray}
	\nonumber
	\E L(\btheta) = nh^d \left(
		\int_{-1}^{1} K(t) \bPsi(t)^\top \btheta \cdot  f(x_0 + th) dt - 
		\int_{-1}^{1} K(t) \exp(\bPsi(t)^\top \btheta) dt
	\right)
\end{eqnarray}
We would like to prove that \( \btheta^{\ast} \approx \btheta^{\circ} \).

Denote \( g(\btheta) = (nh^d)^{-1} \E L(\btheta) \), \( f_0 = f(x_0) \).

Then
\begin{eqnarray}
	\nonumber
	g(\btheta^{\circ}) &=&
		\int_{-1}^{1} K(t) f(x_0 + ht) \cdot \log f_0 dt - \int_{-1}^{1} K(t) 
		f_0 dt
	\\
	\nonumber
	& \geq &
	f_0 I_k \Big(
		(1 - c_{f,h}) \log f_0 - 1
	\Big)
	\enspace .
\end{eqnarray}
From the other side, since for each \( c > 0 \) holds
\begin{equation}
	cx - \exp(x) \leq c \log c - c
	\enspace ,
\end{equation}
then it holds
\begin{eqnarray}
	\nonumber
	g(\btheta^{\ast})
	& = &
	\int_{-1}^{1} K(t) \left(
		\bPsi^\top \btheta f(x_0 + ht) - \exp(\bPsi^\top \btheta)
	\right) dt \\
	& \leq & 
	\int_{-1}^{1} K(t)
		f(x_0 + ht) \big(\log f(x_0 + ht) - 1\big)
	dt \enspace .
\end{eqnarray}

Since the function
$
	\varphi(\tau) = \tau(\log \tau - 1)
$
is unimodal, for any \( \tau^{-}, \tau^{+} \) and \( \tau \in [\tau^{-}, 
\tau^{+}] \) it holds
\begin{equation}
	\varphi(\tau) \leq \max\{\varphi(\tau^{-}), \varphi(\tau^{+})\} \enspace .
\end{equation}
Therefore,
\begin{eqnarray}
	\nonumber
	g(\btheta^{\ast}) & \leq &
	\max \left\{
		(1 + c_{f, h}) \int_{-1}^{1} K(t) f(x_0)
		\big(
			\log f(x_0) + \log (1 + c_{f, h}) - 1
		\big) dt,
	\right. 
	\\&&
	\left.
		(1 - c_{f, h}) \int_{-1}^{1} K(t) f(x_0)
		\big(
			\log f(x_0) + \log (1 - c_{f, h}) - 1
		\big) dt
	\right\}
	\\
	\nonumber
	&=& f_0 I_k (1 \pm c_{f,h}) \Big(
		(\log f_0 - 1) + \log(1 \pm c_{f,h})
	\Big) \enspace ,
\end{eqnarray}
\begin{equation}
	g(\btheta^{\ast}) - g(\btheta^{\circ}) \leq
	f_0 I_K \Big(
		(1 \pm 1) c_{f,h} \log f_0 \mp c_{f,h} +
		(1 \pm c_{f,h}) \log(1 \pm c_{f,h})
	\Big)
\end{equation}
We see that the difference between \( g(\btheta^{\circ}) \) and \( 
g(\btheta^{\ast}) \) is 
small because \( \log (1 \pm c_{f,h}) \) is of order \( c_{f,h} \), which is of 
order \( O(h) \), \( h \to 0 \).

From the Taylor expansion, we have for some \( \bar \btheta \):
\begin{equation}
	g(\btheta^{\circ}) = g(\btheta^{\ast}) + \nabla g(\btheta^{\ast})^\top 
	(\btheta^{\circ} - \btheta^{\ast}) - \dfrac12 \| d_0^2 (\bar \btheta) 
	(\btheta^\circ - \btheta^{\ast}) \|^2 \enspace , 
\end{equation}
\begin{eqnarray}
	\nonumber
	\| d_0 (\bar \btheta) (\btheta^{\circ} - \btheta^{\ast}) \|^2
	& = &
	2 (g(\btheta^{\ast}) - g(\btheta^{\circ})) \\
	\nonumber
	&\leq&
	2 f_0 I_K \Big(
			(1 \pm 1) c_{f,h} \log f_0 \mp c_{f,h} +
			(1 \pm c_{f,h}) \log(1 \pm c_{f,h})
		\Big) \\
	&=& f_0 \cdot \phi_1^2 \enspace .
\end{eqnarray}
	
Now we are going to perform the trick, which has a bit of asymptotical and 
implicit flavour. Let \( \btheta_1, \btheta_2 \in [\btheta^{\circ}, 
\btheta^{\ast}] \)~--- two points on the segment with the ends \( 
\btheta^{\circ} 
\) and \( \btheta^{\ast} \).
\begin{equation}
	\eps^2 = \sup_{\btheta_1, \btheta_2} \| \mathbf I_p - D^{-1} (\btheta_1) 
	D^{2} 
	(\btheta_2) D^{-1} (\btheta_1) \|
\end{equation}	
Note that the following chain of inequalities is satisfied:
\begin{eqnarray}
	\nonumber
	\varepsilon^2 & \leq &
	\left(
	\sup_{t \in[-1, 1]}\exp(\bPsi^\top (t) d_0^{-1} d_0 (\btheta_1 - 
	\btheta_2)) - 1
	\right)^2
	 \\
	\nonumber
	&\lesssim& \sup_{t \in [-1, 1]} \bPsi(t)^\top d_0^{-2} \bPsi(t) \cdot 
	(\btheta_1 - \btheta_2)^{\top} d_0^2 (\btheta_1 - \btheta_2)
	\\
	\nonumber
	& \leq &
	(1 + \varepsilon) \sup_{t \in [-1, 1]} \bPsi(t)^\top 
	d_0^{-2}(\btheta^{\circ}) \bPsi(t) \cdot (1 + \varepsilon) 
	(\btheta^{\circ} - \btheta^{\ast})^{\top} d_0^2(\bar \btheta) 
	(\btheta^{\circ} - \btheta^{\ast}) \\
	\nonumber
	& \leq &
	(1 + \eps)^{2} \c_1^2 \phi_1^2 \enspace .
\end{eqnarray}
The terms \( f_0 \) and \( f_0^{-1} \) cancelled out, because they appear both 
in \( \| d_0(\bar \btheta) (\btheta^{\circ} - \btheta^{\ast}) \| \) and \( 
\sup_{t} \bPsi(t)^\top d_0^{-2}(\btheta^{\circ}) \bPsi(t) = f_0^{-1} \c_1^2 \).
Therefore,
\begin{equation}
	\eps \leq (1 + \eps) \c_1 \phi_1, \qquad
	\eps \leq \dfrac{\c_1 \phi_1}{1 - \c_1 \phi_1} = O(h^{1/2}) \enspace .
\end{equation}
We can also notice that for any vector \( \vec v \) it holds
\begin{equation}
	\vec v^\top d_0(\btheta^\circ) \vec v \leq (1 + \eps) \vec v^\top 
	d_0(\bar \btheta) \vec v \enspace ,
\end{equation}
therefore
\begin{equation}
	\| d_0(\btheta^{\circ}) (\btheta^{\circ} - \btheta^{\ast}) \|^2 \leq
	(1 + \eps) f_0 \phi_1^2 \lesssim f_0 \phi_1^2 (1 - \c_1 \phi_1)^{-1} 
	\enspace .
\end{equation}
\textbf{End of the proof of lemma \ref{lemma:smb}.}
\end{proof}

\begin{remark}
	We have used the symbol ``\( \lesssim \)'' and an approximation ``\( 
	\exp(t) - 1 \lesssim t \)''. This means that if \( t \) is close to zero, 
	then the 
	function \( \exp(t) - 1 \) is bounded by some linear function \( 
	\mathrm{lin} (t) = \mathrm{coeff} \cdot t \). It is enough to guarantee 
	that \( \btheta^{\circ} - \btheta^{\ast} \) is small when \( h \to 0 \). It 
	is true because the matrix \( d_0^2 (\btheta^{\circ}) \) is 
	positive-definite and continuous at 
	the vicinity of \( \btheta^{\circ} \) and the difference \( 
	g(\btheta^{\circ}) - g(\btheta^{\ast}) \) is close to zero.	
\end{remark}

After describing the closeness of \( D_n^2(\btheta^{\circ}), D_n^{2} 
(\btheta^{\ast}) \) we describe the closeness of \( D_n^2(\btheta^{\circ}) \) 
and 
\( D_n^2(\btheta^{\bullet}) \), and \( \btheta^{\bullet} \approx 
\btheta^{\circ} \) in metric generated by curvature matrix \( 
d_0^2(\btheta^{\circ}) \).

\begin{lemma}
\label{lemma:constant_approximation}
Let \( I_K = \int_{-1}^{1} K(t) dt \leq 2^d \).
For vectors \( \btheta^{\circ}, \btheta^{\bullet} \) the following holds:
\begin{eqnarray}
		\| \mathbf I_p - D(\btheta^{\bullet}) D^{-2} (\btheta^{\circ}) 
		D(\btheta^{\bullet}) \|
		& \leq &
		(1 + c_{f,h}) B_{p,h} - 1	\enspace , \\
		\| d_0(\btheta^{\circ}) (\btheta^{\circ} - \btheta^{\bullet}) \|^2
		& \leq &
		I_K \big(f(x_0)\big)^3 (c_{f,h} - \log B_{p,h})^2
		\enspace .
\end{eqnarray}
\end{lemma}

\begin{proof}
Denote \( \varphi(x) = \log f(x) \), \( t = (x-x_0) / h \).
According to the lemma \ref{lemma:eigenvalues}, the quantity can be bounded by
\begin{eqnarray}
	\nonumber
	\exp\left( \bPsi^\top(t) (\btheta^{\bullet} - \btheta^{\circ}) \right) - 1 
	&=&
	\exp\left( 		
		(x-x_0) \varphi'(x_0) + 
		\dfrac{(x-x_0)^2}{2!} 
		\varphi''(x_0) +
	\right.
	\\ \nonumber
	&&
	\left.
		\ldots + \dfrac{(x-x_0)^{p-1}}{(p-1)!} \varphi^{(p-1)}(x_0)
	\right) - 1 \\ \nonumber
	& \leq & \exp \left(
		\log f(x) - \log f(x_0) - \log B_{p,h}
	\right) - 1 \\ \nonumber
	& \leq & (1 + c_{f,h}) B_{p,h}^{-1} - 1 \enspace .
\end{eqnarray}
Next,
\begin{eqnarray}
	\nonumber
	\| d_0(\btheta^{\circ}) (\btheta^{\bullet} - \btheta^{\circ}) \|^2
	& = &
	f(x_0)
	\int_{-1}^{1} K(t)
	\Big[
		(\btheta^{\bullet} - \btheta^{\circ})^\top \bPsi(t)
	\Big]
	\Big[
		(\btheta^{\bullet} - \btheta^{\circ})^\top \bPsi(t)
	\Big]^{\top} dt
	\\
	\nonumber
	& = &
	f(x_0)
	\int_{-1}^{1} K(t)
	\Big[
		f(x) - f(x_0) - \log B_{p,h}
	\Big]^2 dt
	\\
	& \leq &
	I_k f(x_0)^3 (c_{f,h} - \log B_{p,h})^2 \enspace .
\end{eqnarray}
\textbf{End of the proof of lemma \ref{lemma:constant_approximation}.}
\end{proof}

\begin{remark}
	Since \( c_{f,h} = O(h) \), \( B_{p,h} = 1 + O(h^p) \) the right hand side 
	is of order \( O(h) \).
\end{remark}

\begin{lemma}
\label{lemma:gradient_difference}
Let \( \mathrm{pr}_2(x_0) = \int_{-1}^{1} K((x-x_0)/h) f^2(x) dx \), 
\( g(\btheta) = (nh^d)^{-1} \E L(\btheta) \).

Then the following inequality holds:
	\begin{equation}
		\| d_0^{-1}(\btheta^{\circ}) (\nabla g(\btheta^{\ast}) - \nabla 
		g(\btheta^{\bullet})) \|^2 \leq p \cdot (1 - B_{p,h})^2  
		h^{-d}\mathrm{pr}_2 (x_0)
		\enspace .
	\end{equation}
\end{lemma}

\begin{proof}
	Let \( x = x_0 + t \cdot h \).
	Since \( \nabla g(\btheta^{\ast}) = 0 \), the difference between 
	gradients is equal to
\begin{eqnarray}
	\nabla g (\btheta^\bullet)
	&=&
	\int_{\X} K(t) \bPsi(t) [f(x) - \exp(\bPsi^\top \btheta^\bullet)] dt
	\enspace .
\end{eqnarray}
Denote \( \delta(t) = \sqrt{K(t)} (f(x) - \exp(\bPsi^\top(t) 
\btheta^{\bullet})) \), and \(\vec \psi(t) =  \sqrt{K(t)} \bPsi(t) \).  
After applying 
lemma 
\ref{lemma:cauchy_for_matrices} (analog of the 
Cauchy-Schwarz inequality for vectors and matrices), we obtain:
\begin{equation}
	\int_{-1}^{1} \vec \psi(t) \delta(t) dt
	\int_{-1}^{1} \vec \psi^{\top}(t) \delta(t) dt \leq
	\int_{-1}^{1} \vec \psi(t) \vec \psi^\top (t) dt
	\int_{-1}^{1} \delta^2(t) dt
	\enspace .
\end{equation}
Thus,
\begin{equation}
	\nabla g(\btheta^{\bullet}) \nabla g(\btheta^{\bullet})^{\top} \leq 
	d_0^2(\btheta^{\circ}) 
	\int_{-1}^{1}\delta^2(t) dt 
	\enspace .
\end{equation}
This allows to finish the proof:
\begin{eqnarray}
	\| d_0^{-1}(\btheta^{\circ}) \nabla g(\btheta^{\bullet}) \|^2
	&=& 
	\mathrm{Tr} \;
		\nabla g(\btheta^{\bullet})^\top 
		d_0^{-2}(\btheta_{\circ})
		\nabla g(\btheta^{\bullet})
	\\
	& = &
	\mathrm{Tr} \;
		\nabla g(\btheta^{\bullet})
		\nabla g(\btheta^{\bullet})^\top
		d_0^{-2}(\btheta^{\circ})
	\\
	& \leq &
	p\int_{-1}^{1}\delta^2(t) dt
	\enspace .
\end{eqnarray}
The integral is bounded using the bias definition from section 
\ref{section:small_bias}:
\begin{eqnarray}
	\int_{-1}^{1} \delta^2(t) dt 
	& = &
	\int_{-1}^{1} K(t) [f(x) - \exp(\bPsi^\top 
	\btheta^{\bullet})]^2 
	dt
	\\
	& \leq &
	\int_{-1}^{1} K(t) f(x)^2 [1 - B_{p,h}]^2 dt
	\\
	& = &
	(1 - B_{p,h})^2 \cdot h^{-d}\mathrm{pr}_2 (x_0)
	\enspace .
\end{eqnarray}
\textbf{End of the proof of lemma \ref{lemma:gradient_difference}.}
\end{proof}

\subsection{Facts from Linear Algebra}
\label{section:linalg}
 
\begin{lemma} \label{lemma:eigenvalues}
If matrices $A, B \in \R^{p \times p}$ have the form
\begin{equation}
    A^2 = \int_{\X} \bPsi(x) \lambda_A(x) \bPsi(x)^\top dx\enspace , \quad
    B^2 = \int_{\X} \bPsi(x) \lambda_B(x) \bPsi(x)^\top dx
    \enspace ,
\end{equation}
and $\lambda_A(x), \lambda_B(x) > 0$, then the eigenvalue set of the quotient 
$B^{-1}A^2 B^{-1}$ belongs to the interval
\begin{equation}
    \left[
    \min_{x \in \X} \frac{\lambda_A(x)}{\lambda_B(x)},
    \max_{x \in \X} \frac{\lambda_A(x)}{\lambda_B(x)}
    \right]
    \enspace.
\end{equation}
\end{lemma}
\begin{proof}
Let us introduce self-adjoint operators $\Lambda_A, \Lambda_B \colon L_2[\X] 
\to L_2[\X]$, $\bPsi \colon \R^p \to L_2[\X]$:
\begin{equation}
    \Lambda_A f(x) = \lambda_A(x) f(x)\enspace, \quad
    \Lambda_B f(x) = \lambda_B(x) f(x)\enspace, \quad    
    \bPsi \bf v = \bPsi(x)^\top \bf v\enspace.
\end{equation}
The scalar product takes form
\begin{equation}
    \Big\langle g(x), \Lambda f(x) \Big\rangle = \int_{\X} g(x) \lambda (x) 
    f(x) dx\enspace .
\end{equation}
Then matrices $A^2, B^2$ can be rewritten in the form
\begin{equation}
    A^2 = \bPsi^\ast \Lambda_A \bPsi\enspace , \quad
    B^2 = \bPsi^\ast \Lambda_B \bPsi\enspace .
\end{equation}
It is not difficult to check that $\Lambda_B \succeq \min_{x \in \X} 
\frac{\lambda_B(x)}{\lambda_A(x)} \Lambda_A$. Therefore, for operator $\bPsi$ 
it holds
\begin{equation}
    \bPsi^\ast \Lambda_B \bPsi \succeq \min_{x \in \X} 
        \frac{\lambda_B(x)}{\lambda_A(x)} \bPsi^\ast \Lambda_A \bPsi        
        \enspace , \quad
        \lambda_{\max}(B^{-1}A^2 B^{-1}) \geq \max_{x \in \X} 
        \frac{\lambda_A(x)}{\lambda_B(x)}
        \enspace.
\end{equation}
Similar argument is suitable for the lower bound.

\noindent
\textbf{End of the proof of lemma \ref{lemma:eigenvalues}.}
\end{proof}

\begin{lemma}
\label{lemma:cauchy_for_matrices}
Let \( \vec \psi(t) \colon [-1, 1] \to \R^{p} \) be some vector-valued 
integrable 
function, \( 
\delta(t) \colon [-1, 1] \to \R \)~--- integrable function. Then 
the following matrix inequality holds:
\begin{equation}
	\int_{-1}^{1} \vec \psi(t) \delta(t) dt
	\int_{-1}^{1} \vec \psi^{\top}(t) \delta(t) dt
	\leq
	\int_{-1}^{1} \vec \psi(t) \vec \psi^{\top}(t) dt
	\int_{-1}^{1} \delta^2(t) dt
\end{equation}
\end{lemma}

\begin{proof}
	Consider the matrix-valued non-negative integral:
	\begin{equation}
		I = \int_{-1}^{1} d \tau \int_{-1}^{1} dt
		\left[
			\vec \psi(\tau) \delta(t) - \vec \psi(t) \delta(\tau)
		\right]
		\left[
			\vec \psi(\tau) \delta(t) - \vec \psi(t) \delta(\tau)
		\right]^{\top} \geq 0
	\end{equation}
	Denote \( \int_{-1}^{1} \vec \psi(t) \delta(t) dt = A \), \( \int_{-1}^{1} 
	\vec \psi(t) \vec \psi^{\top}(t) dt = M \), \( \int_{-1}^{1} \delta^2(t) dt 
	= D 
	\).
	
	The integral \( I \) can be expanded and rewritten as:
	\begin{eqnarray}
		I &=& \int_{-1}^{1} \int_{-1}^{1} \Big[
			\delta(t)^{2} \vec \psi(\tau) \vec \psi^{\top}(\tau) + 
			\delta(\tau)^{2} \vec \psi(t) \vec \psi^{\top}(t)
		\\
		&&-
		\delta(t) \delta(\tau) \vec \psi(t) \vec \psi(\tau)^{\top} - 
		\delta(t) \delta(\tau) \vec \psi(\tau) \vec \psi(t)^{\top}
		\Big]\\
		&=& 2A A^\top - 2 D M
	\end{eqnarray}
	Therefore,
	\begin{equation}
		A A^\top \leq D M
	\end{equation}
\textbf{End of the proof of lemma \ref{lemma:cauchy_for_matrices}.}
\end{proof}

\begin{lemma}
\label{lemma:legendre}
Let \( \bPsi(t) = (1, t, t^2, \ldots, t^{p-1})^\top \). Consider the matrix
\begin{equation}
	A^2 = \int_{-1}^{1} \bPsi(t) \bPsi^\top(t) dt \enspace .
\end{equation}
Then the polynomial defined by
\begin{equation}
\label{eq:polynomial_definition}
	P(t) = \bPsi^\top(t) A^{-2} \bPsi(t)
\end{equation}
attains its maximal value at points \( t = \pm 1 \), and this value equals to 
\( p^2 / 2 \). Moreover, the fact is still valid if we consider \( \bPsi(t) 
= (1, 
P_1(t), P_2(t), \ldots, P_{p-1}(t))^{\top} \), where the polynomials \( P_i(t) 
\) have degrees less than \( p \) and form a basis in the space of polynomials 
with degree less than \( p \).
\end{lemma}

\begin{remark}
We formulated the following lemma experimentally and the proof for our guess 
was kindly presented by Ilya Bogdanov \cite{Bogdanov} at 
Mathoverflow. Some 
interesting properties of the polynomial \( P(t) \) are listed in the 
discussion. The claim can be probably generalized to higher-dimensional case, 
but we still don't know whether it is possible to treat non-uniform kernel case 
efficiently. The shape of the polynomial \( P(t) \) suggests that if we 
``suppress'' its behaviour at the tails, say by choosing appropriate kernel 
function, this constant can be reduced significantly.
\end{remark}

\begin{proof}
It is well-known that Legendre 
polynomials 
form the 
orthogonal basis for the space of polynomials, defined on the segment \( [-1, 
1] \) with respect to the scalar product
\begin{equation}
	\langle f, g \rangle = \int_{-1}^{1} f(t) g(t) dt \enspace .
\end{equation}
The ordinary Legendre polynomials \( L_k(t) \) are equal to \( \pm 1 \) at the 
ends of the interval \( [-1, 1] \), and their scalar product equals to
\begin{equation}
	\langle L_{i}(t), L_j(t) \rangle = \delta_{ij} \cdot \dfrac{2}{2j + 1} 
	\enspace .
\end{equation}
We consider their normed versions \( \tilde L_j(t) = \sqrt{\dfrac{2j + 1}{2}} 
L_j(t) \) so that the basis \( \vec L = (L_0, L_1, \ldots, L_{p-1}) \) is 
orthonormal, i.e. \( \int_{-1}^{1} \vec L(t) \vec L^\top(t) dt = \mathbf I_p \).

Since the polynomials \( 1, t, t^2, \ldots, t^{p-1} \) have degrees less than 
\( p \) and are linearly independent, the basis \( \bPsi(t) \) can be 
transferred 
into the Legendre polynomial basis \( \vec L(t) \) 
with some transition matrix \( S \): \( \bPsi = S \vec L \).
Substituting this value into the expression (\ref{eq:polynomial_definition}), 
we 
obtain:
\begin{eqnarray}
	P(t) &=& \vec L^\top(t) S^\top \left(\int_{-1}^{1} S \vec L (t) \vec 
	L^{\top} (t) S^{\top} dt \right)^{-1} S \vec L(t) \\
	&=& \vec L^\top (t)  S^{\top}S^{-\top} \mathbf I_p S^{-1} S \vec L(t) \\
	&=& \vec L^\top(t) \vec L(t)
	= \sum_{j=0}^{p-1} \tilde L_j^2(t) \enspace .
\end{eqnarray}
It is well-known that Legendre 
polynomials \( L_j(t) \) are uniformly 
bounded \( |L_j(t)| \leq 1 \), and the maximum is attained at \( t = \pm 1 \). 
Therefore, the maximal value of the polynomial \( P(t) \) on the segment \( 
[-1, 1] \) equals to
\begin{equation}
	P(\pm 1) = \dfrac12 \sum_{j=0}^{p-1} (2j+1) = \dfrac{p^2}{2} \enspace .
\end{equation}
\textbf{End of the proof of lemma \ref{lemma:legendre}.}
\end{proof}

\begin{lemma}
\label{lemma:relation_of_core_constants}
Let \( \mathrm{pr}_1(x_0) = \int_{-1}^{1} K((x-x_0)/h) f(x) dx \),
\begin{equation}
	C_{V, f}^2 = nh^{d} \sup_{t \in [-1,1]} K(t)^2\bPsi^\top(t) V_n^{-2}(f(x)) 
	\bPsi(t)
	\enspace .
\end{equation}
Then we have an \emph{exact} relationship:
\begin{equation}
	C_{V, f}^{2} = \sup_{t\in[-1,1]}K(t)^2 \bPsi(t)^\top \left[\int K f \bPsi 
	\bPsi^\top 
	d\tau\right]^{-1} \bPsi(t) + \dfrac{h^d K(t)^2}{1 - \mathrm{pr}_1(x_0)}
	\enspace ,
\end{equation}
which leads to the inequality
\begin{equation}
	C_{V, f}^{2} \leq (1 - c_{f,h})^{-1} f(x_0)^{-1} \c_2^2 + \dfrac{h^d}{1 - 
	\mathrm{pr}_1(x_0)} 
	\enspace .
\end{equation}
\end{lemma}

\begin{proof}
By the definition of \( V_n(f(x)) \), we can express
\begin{equation}
	n^{-1} V_n^{2}(f(x)) = \int_{-1}^{1} K f \bPsi \bPsi^\top dx - 
	\int_{-1}^{1} Kf \bPsi dx \int_{-1}^{1} Kf \bPsi^{\top} dx
	\enspace ,
\end{equation}
and replacing \( dx \) with \( h^{d} dt \), we obtain
\begin{equation}
	(nh^{d})^{-1} V_n^2(f(x)) = \int_{-1}^{1} Kf \bPsi \bPsi^\top dt
	- h^d \int_{-1}^{1} Kf \bPsi dt \int_{-1}^{1} Kf \bPsi^{\top} dt
	\enspace ,
\end{equation}
which can be denoted as \( A - h^d \vec u \vec u^\top \) with \( A = \int 
Kf\bPsi \bPsi^\top dt \), \( \vec u = \int K f \bPsi dt \).
Then we apply Sherman--Morrison formula for 
one-rank inverse matrix updates:
\begin{equation}
	(A - \lambda \vec u \vec u^\top)^{-1} = A^{-1} + \lambda \dfrac{A^{-1} \vec 
	u \vec 
	u^{\top} A^{-1}}{1 - \lambda
	\vec u^\top A^{-1} \vec u}
	\enspace .
\end{equation}
Since the basis \( \bPsi \) has important property that first element of the 
basis \( \psi_0(t) \) is constant \( 1 \), we note that
\begin{equation}
	\vec u = A \cdot [1, 0, \ldots, 0]^{\top}, \quad
	A^{-1} \vec u = [1, 0, \ldots, 0] \enspace ,
\end{equation}
and this allows to simplify the above expression. Multiplying by \( \bPsi \) 
from the right and from the left, we obtain final exact expression:
\begin{equation}
	C_{V, f}^{2} = \sup_{t \in [-1, 1]} \bPsi^\top A^{-1} \bPsi + 
	\dfrac{\lambda}{1 - \lambda \int K(t) f(x_0 + ht) dt} \enspace ,
\end{equation}
where \( \lambda = h^d \). The consequent inequality in straightforward.

\textbf{End of the proof of lemma \ref{lemma:relation_of_core_constants}.}
\end{proof}

\subsection{Deviation Bounds for Quadratic Forms}

\begin{lemma}[Laurent and Massart, \cite{Massart}]
\label{lemma:quadratic}
Let $A \in \R^{p \times p}$, $\bxi \in \R^{p}$ be a random vector of the form 
$\bxi = A^{-1} \vec{u}$, where $\E [\vec u] = 0$. Denote $V^2 = \Var \vec u$, 
$\a^2 = \lambda_{\max}(A^{-1}VA^{-1})$.

Suppose that the vector $V^{-1} u$ satisfies the condition
\begin{equation}
    \log \E \exp (\vec {\gamma^\top} V^{-1} \vec u) \leq
    \frac{\nu_0^2\|\vec{\gamma}\|^2}{2}, \qquad \vec{\gamma} \in \R^p
    \enspace .
\end{equation}
Then for each $\z > 0$, \( \z \leq \mathfrak g^2 / 4 \)
\begin{equation}
    \P(\| \bxi \| > \zeta(p, \z)) \leq 2 e^{-\z} + 8.4 e^{-\mathfrak g^2 / 4}
    \enspace ,
\end{equation}
where $\zeta(p, \z)$ is defined by
\begin{equation}
    \zeta(p,x) = \a \nu_0 (\sqrt p + \sqrt{2\z})
    \enspace .
\end{equation}

\end{lemma}

\textbf{Acknowledgments.} First of all, I would like to thank Vladimir Spokoiny 
for providing the formulation of the problem considered in this article. I 
would also like to thank Fedor Goncharov, Ekaterina 
Krymova, Alexey Balitsky, Alexander Cigler and Nazar Buzun for fruitful 
discussions and providing 
helpful comments. 
I thank Elena Chernousova for pointing out some 
errors and for reviewing this paper, because her remarks helped me to improve 
it. I am grateful to Ilya Bogdanov 
\cite{Bogdanov}, who pointed me the proof of lemma \ref{lemma:legendre} at 
mathoverflow.net and to the whole this community.

%
%

\end{document}